\documentclass[11pt,reqno,a4paper]{amsproc}
\usepackage[T1]{fontenc}
\usepackage[utf8]{inputenc}
\usepackage{amsmath,amssymb,mathtools}
\usepackage[a4paper,left=20mm,right=20mm,top=22mm,bottom=22mm]{geometry}
\usepackage{microtype}
\usepackage[colorlinks=true,linkcolor=blue,citecolor=blue,urlcolor=blue]{hyperref}
\newtheorem{theorem}{Theorem}[section]
\newtheorem{maintheorem}{Theorem}

\newtheorem{maincorollary}[maintheorem]{Corollary}
\newtheorem{proposition}[theorem]{Proposition}
\newtheorem{lemma}[theorem]{Lemma}

\theoremstyle{definition}

\theoremstyle{remark}

\DeclareMathOperator{\rank}{rank}
\DeclareMathOperator{\Sym}{Sym}

\DeclareMathOperator{\fix}{fix}
\DeclareMathOperator{\tr}{tr}
\DeclareMathOperator{\CAT}{CAT}
\newcommand{\F}{\mathbb F}
\newcommand{\R}{\mathbb R}
\newcommand{\Z}{\mathbb Z}
\newcommand{\id}{\mathrm{id}}
\newcommand{\dH}{d_{\mathrm H}}
\newcommand{\rr}{\{0,\ldots,r-1\}}

\title{An example of a non-sofic hyperbolic group}
\author{Harjas Sandhu}
\address{Harjas Sandhu, The University of Queensland, St Lucia, QLD 4072, Australia}
\email{harjas.sandhu@student.uq.edu.au}
\author{Andreas Thom}
\address{Andreas Thom, Faculty of Mathematics, TU Dresden,
01062 Dresden, Germany}
\email{andreas.thom@tu-dresden.de}
\date{7 October 2026}

\begin{document}
\begin{abstract}
We construct a torsion-free, two-dimensional, word-hyperbolic group which is not sofic.
Our construction builds on OpenAI's construction of a non-residually finite
hyperbolic group \cite{OpenAI}.
\end{abstract}

\maketitle

\section{Introduction}

For a positive integer $n$, let $\Sym(n)$ denote the symmetric group on
$\{1,\ldots,n\}$. Its normalized Hamming metric and the fixed-point
fraction of a permutation are given by
\[
 \dH(\sigma,\rho)=\frac{|\{a:\sigma(a)\ne\rho(a)\}|}{n},
 \qquad
 \fix(\sigma)=\frac{|\{a:\sigma(a)=a\}|}{n}.
\]
We write $\id$ for the identity permutation. A group $G$ is sofic if, for every finite subset $F\subset G$ containing
$1$ and every $\eta>0$, there are a positive integer $n$ and a map
$\phi:F\to\Sym(n)$, sending $1$ to the identity permutation, such that
\[
 \dH(\phi(a)\phi(b),\phi(ab))<\eta\quad(a,b,ab\in F),
 \qquad
 \fix(\phi(a))<\eta\quad(a\in F\setminus\{1\}).
\]
Thus multiplication of elements in $F$ is preserved on most points of the finite set,
while each specified nonidentity element fixes only a small proportion
of the points. The notion was introduced by Gromov \cite{Gromov}; see
also Weiss \cite{Weiss}. All amenable groups and all residually finite
groups are sofic; see \cite{Pestov}. The question whether all groups are
sofic was open for a long time.

The first proof of the existence of a non-sofic group was given by
OpenAI \cite{OpenAISofic}. Kun and the second author \cite{KT} constructed non-sofic
generalized wreath products and doubles of residually finite groups.
Chapman and Peled \cite{CP} gave a conditional construction of a
non-sofic hyperbolic group using random simplicial complexes. The aim
of this paper is to construct a torsion-free word-hyperbolic group
which is not sofic. We build on OpenAI's construction of a non-residually
finite torsion-free hyperbolic group \cite{OpenAI}.
Recall that a finitely generated group is
word-hyperbolic if geodesic triangles in a Cayley graph for a finite
generating set are uniformly thin: each side lies within a fixed
distance of the other two sides, see \cite{BH} for details.

We obtain the group as the fundamental group of a finite complex.
By a triangle complex we mean a two-dimensional cell complex obtained
from a graph by attaching triangles, each side being mapped
homeomorphically onto an edge. Several edges may have the same
endpoints. In our construction, each closed face is embedded. We assign a length to every edge and give the faces the metrics of
triangles in the hyperbolic plane of curvature $-1$, with the prescribed
side lengths. The complex then carries the induced path metric.

A geodesic metric space is $\CAT(-1)$ if distances between points on
the sides of a geodesic triangle are at most the distances between the
corresponding points of a triangle in the hyperbolic plane with the same
side lengths. Corresponding points have the same distances from the
endpoints of the corresponding sides. The metric is locally $\CAT(-1)$
if each point has a neighbourhood which, with its intrinsic path metric,
is $\CAT(-1)$. We refer to \cite{BH} for these notions and their basic
properties. Our main result is the following.

\begin{maintheorem}\label{thm:main}
There exists a finite connected triangle complex $K$ with a
locally $\CAT(-1)$ metric such that $G=\pi_1(K)$ is torsion-free,
word-hyperbolic, and non-sofic.
\end{maintheorem}

Put $q=9\cdot10^{22}+1$. The presentation of $G$ in Section~\ref{sec:rank} has
\[
 \begin{aligned}
 10^{12}(q-1)q^4+6\cdot10^{17}q^5-2&<4\cdot10^{132}
 &&\text{generators},\\
 10^{12}(q-1)q^5-10^7(q-1)^5&<6\cdot10^{149}
 &&\text{relations}.
 \end{aligned}
\]

The group $G$ in Theorem~\ref{thm:main} admits no proper cocompact action
by cubical automorphisms on a $\CAT(0)$ cube complex.
Indeed, by Agol's theorem \cite[Theorem 1.1]{Agol} and the work of Haglund and
Wise \cite[Theorem 4.4]{HW}, a word-hyperbolic group admitting such an
action is residually finite and therefore sofic.

\begin{maincorollary}\label{cor:no-sofic-quotient}
There exists a nontrivial torsion-free word-hyperbolic group which has
no nontrivial sofic quotient.
\end{maincorollary}
\begin{proof}
If every nontrivial torsion-free word-hyperbolic group had a nontrivial
sofic quotient, the theorem of Kapovich and Wise
\cite[Theorem 3.4]{KW}, applied to the class of sofic groups, would
imply that every torsion-free word-hyperbolic group is residually sofic
and hence sofic by \cite{ES}. This contradicts Theorem~\ref{thm:main}.
\end{proof}

The construction starts with a family of affine lines in $\F_q^5$,
where $\F_q$ is the field with $q$ elements and $q$ is prime. Its
incidence graph has one vertex for each point
and each line, and an edge whenever the point lies on the line. This
is the graph $D(5,q)$ of Lazebnik and Ustimenko \cite{LU}. Its girth,
the length of its shortest simple cycle, is at least ten
\cite[Proposition 2.1(iii)]{LUW}. In Section~\ref{sec:lines}, we select
a family of lines with distinct distinguished points, called marks,
and delete the incidence with its own mark for each line in a specified
subfamily.

The retained incidences determine a triangle complex with three vertices.
Each retained line gives an array of edges between two of these vertices.
Choosing two rows and two columns gives a closed four-edge path,
traversed with alternating orientations.

In Section~\ref{sec:rank}, we combine polynomial interpolation with the
rank estimate of Alon--Hall--Knauer--Pinchasi--Yuster
\cite[Theorem 6]{AHKPY}. We prove that every assignment of permutations
to the oriented edges has either a triangular boundary word moving a
fixed positive fraction of points or one of the four-edge words fixing
such a fraction. The bound is independent of the degree of the permutations.

In Section~\ref{sec:geometry}, we use the girth bound to construct a
locally $\CAT(-1)$ metric and prove that the four-edge paths are
non-nullhomotopic. Soficity would allow both fractions in the preceding
paragraph to be arbitrarily small. Section~\ref{sec:sofic} derives the
contradiction.

\subsection*{Acknowledgments}
This paper is based on an AI-generated construction obtained by Harjas
Sandhu on October 7, 2026, which was in turn based on the OpenAI breakthrough \cite{OpenAI}. The second author worked through the proof to make it readable
and to reduce unnecessary overhead. OpenAI's GPT Astra helped draft this document. We simplified various constructions and related them to previous work. The aim of this note is to give a reasonably self-contained presentation of the result. An independent AI-assisted proof was given by Douglas--Mghirbi \cite{DM}.

From the construction in \cite{OpenAI} we use the marked lines,
triangular relations associated to Cartesian products, interpolation
identity, and comparison of real and finite-field ranks. We prove these ingredients
and combine them with the rank theorem and the incidence
relations of $D(5,q)$ to obtain the inequality for approximate
permutation assignments. We thank Ian Agol for pointing out the relevance of Dvir's work \cite{Dvir} to the crucial interpolation argument, see Lemma \ref{lem:interpolation}. We view this document as a commentary on OpenAI 252 \cite{OpenAI}.

\section{Marked lines from $D(5,q)$}\label{sec:lines}

We fix
\[
 r=10^6,\qquad h=3\cdot10^{11},\qquad
 q=h^2+1=9\cdot10^{22}+1,\qquad s=\frac{q-1}{10}.
\]
The integer $q$ is prime\footnote{We have $q-1=2^{22}3^25^{22}$.
Direct modular exponentiation gives
$11^{q-1}\equiv1\pmod q$ and
$\gcd(11^{(q-1)/p}-1,q)=1$ for $p\in\{2,3,5\}$.
These identities force $11$ to have order $q-1$ modulo every prime
divisor $\ell$ of $q$. Hence $q-1$ divides $\ell-1$, so $\ell=q$.}.
We consider two copies of $\F_q^5$, points and labels. We write points as $(t,y,z,w,v)$
and labels as $[a,b,c,d,e]$. The incidence relation between points and labels is defined by
\begin{equation}\label{eq:incidence}
 b-y=at,\qquad c-z=bt,\qquad d-w=ay,\qquad e-v=az.
\end{equation}
These equations define a bipartite graph $D(5,q)$ as in
\cite[Section 2]{LUW}.
By \cite[Proposition 2.1]{LUW}, this graph is $q$-regular, has $q^5$
vertices in each part, and has girth at least ten. These properties also
follow from \cite[Proposition 2.1 and Theorem 3.3]{LU}.

For a fixed label, the incident points form the full affine line
\begin{equation}\label{eq:affineline}
 \ell_{[a,b,c,d,e]}=
 \{(t,b-at,c-bt,d-ab+a^2t,e-ac+abt):t\in\F_q\}.
\end{equation}
Different labels determine different lines. Indeed, the normalized direction
$(1,-a,-b,a^2,ab)$ determines $a,b$, and the point with first coordinate
zero then determines $c,d,e$.

Consider the set of incidence edges satisfying $a=t$. Given a point
$(t,y,z,w,v)$, the condition $a=t$ and the incidence equations uniquely
determine a label by
\[
 a=t,\qquad b=y+t^2,\qquad c=z+bt,\qquad d=w+ty,\qquad e=v+tz.
\]
Thus each point is contained in exactly one of these incidence edges.
Conversely, for a fixed label $[a,b,c,d,e]$, the parametrization
\eqref{eq:affineline} gives exactly one incident point with first
coordinate $t=a$. Thus each label is also contained in exactly one of
these incidence edges. These incidence edges therefore form a perfect matching.
Since each matched pair satisfies \eqref{eq:incidence}, the point matched
to a label lies on the corresponding affine line.
We henceforth index each line by its matched point, writing $\ell_i$
for the line marked at $i\in\F_q^5$.

Put $S=\{0,\ldots,s-1\}\subset\F_q$, using the standard integer
representatives, and set
\[
 I=S^5,\qquad J=(\F_q\setminus\{-1\})\times\F_q^4.
\]
We retain all lines $\ell_i$ with $i\in J$ and distinguish those with
$i\in I$. For each distinguished line, we delete only its incidence
with its own mark. All incidences of the other retained lines remain.
Thus, for $x\in\F_q^5$, the remaining incidences at $x$ are indexed by
\begin{equation}\label{eq:puncture}
 J(x)=\{i\in I:x\in\ell_i\setminus\{i\}\}
 \cup\{i\in J\setminus I:x\in\ell_i\}.
\end{equation}
The bipartite graph $\Gamma$ on $\F_q^5\sqcup J$, whose two parts
represent points and retained lines, has an edge from $x$ to $i$
precisely when $i\in J(x)$. It is a subgraph of $D(5,q)$ and therefore
has girth at least ten.

For a fixed point and a prescribed first label coordinate $a$,
\eqref{eq:incidence} uniquely determines $b,c,d,e$. The first coordinate
of the matched point is $a$, so the retained lines are precisely those
with $a\ne-1$. Thus $q-1$ retained lines pass through each point before deletion. A point loses one incidence if it
belongs to $I$, and none otherwise. Consequently,
\[
 |J(x)|=
 \begin{cases}
  q-2=(h-1)(h+1),&x\in I,\\
  q-1=h\cdot h,&x\notin I.
 \end{cases}
\]
For each point $x$, put $(s_x,t_x)=(h-1,h+1)$ if $x\in I$ and
$(s_x,t_x)=(h,h)$ otherwise. All these factors are at least $r$.
Choose a bijection
\begin{equation}\label{eq:block}
 J(x)\longrightarrow (\Z/s_x\Z)\times(\Z/t_x\Z),
 \qquad i\longmapsto(\alpha_x(i),\beta_x(i)).
\end{equation}
The sets $J(x)$ may
overlap, but their bijections are chosen separately for each point.

The following interpolation argument is adapted from
\cite[Section 2.5, equation (2.13)]{OpenAI} to the grid $S^5$. It uses low-degree polynomials on affine
lines, as in Dvir's proof of the finite-field Kakeya conjecture
\cite{Dvir}, applied here to the entries of rank-one matrices.

\begin{lemma}\label{lem:interpolation}
There are matrices $H_x\in M_I(\F_q)$, indexed by
$x\in\F_q^5$, such that
\begin{equation}\label{eq:certificate}
 \rank H_x\le1,\qquad
 \sum_{x:\,i\in J(x)}H_x=
 \begin{cases}
  -E_{ii},&i\in I,\\
  0,&i\in J\setminus I.
 \end{cases}
\end{equation}
\end{lemma}
\begin{proof}
For $c\in S$, let $p_c$ be the Lagrange basis polynomial associated with
the node $c$ in $S$, given by
\[
 p_c(T)=\prod_{d\in S\setminus\{c\}}\frac{T-d}{c-d}.
\]
It is the unique polynomial of degree at most $s-1$ satisfying
\[
 p_c(d)=
 \begin{cases}
  1,&d=c,\\
  0,&d\in S\setminus\{c\}.
 \end{cases}
\]
For $i=(c_{i1},\ldots,c_{i5})\in I$ and $X=(X_1,\ldots,X_5)$, set
\[
 f_i(X)=\prod_{j=1}^5p_{c_{ij}}(X_j),\qquad
 f=(f_i)_{i\in I},\qquad H_x=f(x)f(x)^{\mathsf T}.
\]
We have $\rank H_x\le1$. For $i,i'\in I$, $f_i(i')$ equals $1$ if $i=i'$
and $0$ otherwise, so $H_i=E_{ii}$ for $i\in I$. Each $f_i$ has total
degree at most $5(s-1)$. It follows that every entry of $ff^{\mathsf T}$,
restricted to an affine line, is a polynomial of degree at most
$10(s-1)=q-11<q-1$.

It is elementary to see that the sum of the values of a polynomial of degree less than
$q-1$ over $\F_q$ is zero. Applying this observation entrywise gives
\[
 \sum_{x\in\ell_i}H_x=0\qquad(i\in J).
\]
For $i\in I$, deleting the term at its mark gives $-H_i=-E_{ii}$.
For $i\in J\setminus I$, the full line is retained, so its sum remains
zero. This proves \eqref{eq:certificate}.
\end{proof}

\section{A finite permutation obstruction}\label{sec:rank}

Our next aim is to define a connected triangle complex whose fundamental group will have all the required properties. We use the construction in Section~\ref{sec:lines} and write $[r]=\rr$.
We follow the construction in \cite[Section 2.3]{OpenAI}.
Consider three vertices $O,V,W$ and define oriented edges
\[
 A_{i,u,v}:V\longrightarrow W,\qquad
 L_{x,u,j}:O\longrightarrow V,\qquad R_{x,v,k}:O\longrightarrow W.
\]
Here $i\in J$, $u,v\in[r]$, $x\in\F_q^5$, $j\in\Z/s_x\Z$, and
$k\in\Z/t_x\Z$. For each point $x$, $i\in J(x)$, and $u,v\in[r]$,
attach a triangle with boundary, read in path order,
\begin{equation}\label{eq:face}
 L_{x,u,\alpha_x(i)+v}\,A_{i,u,v}\,R_{x,v,\beta_x(i)+u}^{-1}.
\end{equation}
Denote the resulting finite triangle cell complex by $K$. There are
$L$- and $R$-edges, and hence $K$ is connected.

Let $G=\pi_1(K,O)$. Choose a point $x_0\in\F_q^5$. The edges
$c_V:=L_{x_0,0,0}$ and $c_W:=R_{x_0,0,0}$ form a spanning tree of the graph of
vertices and edges of $K$. Let $c_O$ be the constant path at $O$.
For an oriented edge $e:A\to B$, define
\begin{equation}\label{eq:edgeelement}
 g_e=[c_Aec_B^{-1}]\in G.
\end{equation}
Denote the elements associated with $A_{i,u,v}$, $L_{x,u,j}$, and
$R_{x,v,k}$ by $a(i,u,v)$, $l(x,u,j)$, and $r(x,v,k)$, respectively.
We multiply loops in path order. Collapsing the spanning tree gives a
finite presentation of $G$ with generators
$
 a(i,u,v), l(x,u,j), r(x,v,k),
$
with the index ranges specified above, and relations
\begin{equation}\label{eq:relation}
 l(x,u,\alpha_x(i)+v)a(i,u,v)
 r(x,v,\beta_x(i)+u)^{-1}=1
\end{equation}
for every point $x$, $i\in J(x)$, and $u,v\in[r]$, together with
\[
 l(x_0,0,0)=1,\qquad r(x_0,0,0)=1.
\]

Our next aim is to understand how this presentation interacts with the assignment of permutations to the generators. We identify a permutation $\sigma\in\Sym(n)$ with
the matrix $P_\sigma$ defined by $P_\sigma e_a=e_{\sigma(a)}$, where
$e_a$ is the $a$-th standard coordinate vector. Thus inverse matrices
are transposes, and $\tr(P_\sigma)=n\fix(\sigma)$.
Assign arbitrary permutations of degree $n$ to the oriented edges of $K$.
Denote the permutations assigned to $A_{i,u,v}$, $L_{x,u,j}$, and
$R_{x,v,k}$ by $A(i,u,v)$, $L(x,u,j)$, and $R(x,v,k)$, respectively.
Reversing an edge replaces its permutation by its
inverse. For a triangular face indexed by $\tau=(x,i,u,v)$, define
\begin{equation}\label{eq:triangledefect}
 \delta_{x,i,u,v}=
 \dH\bigl(L(x,u,\alpha_x(i)+v)A(i,u,v)
 R(x,v,\beta_x(i)+u)^{-1},\id\bigr).
\end{equation}
This is exactly the defect of the corresponding relation being satisfied in the permutation group.
For $i\in I$, $u\ne u'$ and $v\ne v'$, define the rectangular product
\begin{equation}\label{eq:rectangle}
 Q_i(u,u',v,v')=A(i,u,v)A(i,u',v)^{-1}A(i,u',v')A(i,u,v')^{-1}.
\end{equation}
Let $\lambda_i$ be the average of $\fix(Q_i(u,u',v,v'))$ over
all $r^2(r-1)^2$ ordered tuples $(u,u',v,v')$ with $u\ne u'$ and $v\ne v'$.

The key technical result is the following theorem. It gives a lower bound on the sum of defects of the defining relations together with a measure of non-triviality of the rectangular words $w_{i;u,u',v,v'}:=
 a(i,u,v)a(i,u',v)^{-1}a(i,u',v')a(i,u,v')^{-1}$. We will later show that all these rectangular words are non-trivial in $G$.

\begin{theorem}\label{thm:obstruction}
Every such assignment, in every degree $n\ge1$, satisfies
\begin{equation}\label{eq:linear}
 \frac{q^5}{5}\le\sum_\tau\delta_\tau+\frac{r^2}{8}\sum_{i\in I}\lambda_i.
\end{equation}
\end{theorem}

The proof of this theorem occupies the remainder of this section and needs various lemmas.
The following result appears in \cite[Lemma 2.4]{OpenAI}.
It is also a consequence of Gabber's bound on torsion in
cokernels; see \cite[Lemma 1]{Soule} and \cite[Lemma 3.1]{Sauer}.
We include a direct proof.

\begin{lemma}\label{lem:reduction}
If an integer matrix $M$ has row norms at most $\sqrt r$ and $q>r$ is
prime, then $$\rank_{\F_q}M\ge\tfrac12\rank_\R M.$$
\end{lemma}
\begin{proof}
Set $d=\rank_\R M$ and choose a nonsingular $d\times d$ minor $B$ if $d>0$.
By Hadamard's inequality, we have $1\le|\det B|\le r^{d/2}$.
Let $d'=\rank_{\F_q}B$ and choose $d'$ rows forming a basis modulo $q$.
Subtract integer lifts of their linear combinations from the remaining rows.
These row operations preserve the determinant and make $d-d'$ rows
divisible by $q$. Hence $q^{d-d'}$ divides $\det B$, so
\[
 q^{d-d'}\le r^{d/2}.
\]
Since $q>r$, this implies $d'\ge d/2$. If $d=0$, there is nothing to prove.
\end{proof}

\begin{lemma}\label{lem:lower}
For $i\in J$, let $Y_i=(A(i,u,v))_{u,v\in[r]}$, an $rn\times rn$
integer block matrix. For every $i\in I$,
\begin{equation}\label{eq:lower}
 \rank_{\F_q}Y_i\ge
 \frac{r^2n}{2\bigl(2r-1+(r-1)^2\lambda_i\bigr)}
 \ge n\left(\frac r4-\frac{r^2\lambda_i}{8}\right).
\end{equation}
\end{lemma}
\begin{proof}
We apply \cite[Theorem 6]{AHKPY}. For a nonzero $0$-$1$ matrix $A$,
this theorem gives the following estimate in terms of the bipartite graph
whose vertices are the rows and columns and whose edges correspond to
entries equal to $1$:
\begin{equation}\label{eq:publishedrank}
 \rank_\R A\ge\frac{e^2}{\sum_x d_x^2-e+4c_4}.
\end{equation}
Here $e$ is the number of edges, $d_x$ the vertex degrees, and $c_4$ the
number of unoriented four-cycles.

For $A=Y_i$, each of the $2rn$ vertices has degree $r$ and $e=r^2n$.
Every four-cycle uses two different row blocks and two different column
blocks: two rows of one permutation block cannot share a column, and
the analogous assertion holds for columns. For each ordered choice of
these four blocks, the cycles are counted by the fixed points of
\eqref{eq:rectangle}. Each unoriented cycle is counted four times, giving
$
 4c_4=nr^2(r-1)^2\lambda_i.
$
By \eqref{eq:publishedrank}, we get
\[
 \rank_\R Y_i\ge\frac{r^2n}{2r-1+(r-1)^2\lambda_i}.
\]
Each row of $Y_i$ has exactly $r$ entries equal to $1$, so
Lemma~\ref{lem:reduction} proves the first bound in \eqref{eq:lower}.

For the second inequality, use $(1+t)^{-1}\ge1-t$ for $t\ge0$ to obtain
\[
 \frac{r^2}{2\bigl(2r-1+(r-1)^2\lambda\bigr)}
 \ge\frac{r}{4(1+r\lambda/2)}
 \ge\frac r4-\frac{r^2\lambda}{8}\qquad(\lambda\ge0).
\]
\end{proof}

The following lemma extends the exact factorization and rank bounds in
\cite[equations (2.11)--(2.12)]{OpenAI} to permutation assignments with
triangular defects.

\begin{lemma}\label{lem:upper}
Over $\F_q$, let $Z_x=\sum_{i\in J(x)}Y_i$, an $rn\times rn$ matrix.
For every $x\in\F_q^5$,
\begin{equation}\label{eq:upper}
 \rank Z_x\le n+n\sum_{\substack{i\in J(x)\\u,v\in[r]}}\delta_{x,i,u,v}.
\end{equation}
\end{lemma}
\begin{proof}
Permutation matrices $P,Q$ differing on $k$ points have at most $k$
nonzero columns in $P-Q$. Thus $\rank(P-Q)\le n\dH(P,Q)$ over every
field. Bi-invariance gives
\[
 \rank\bigl(A(i,u,v)-L(x,u,\alpha_x(i)+v)^{-1}
 R(x,v,\beta_x(i)+u)\bigr)\le n\delta_{x,i,u,v}.
\]
Fix a point $x$. In $Z_x$, replace each permutation matrix by the
corresponding product in the preceding display. For fixed $u,v$, the
pairs $(\alpha_x(i)+v,\beta_x(i)+u)$, with $i\in J(x)$, traverse the
whole Cartesian product once. The resulting matrix is
\begin{equation}\label{eq:factor}
 \begin{pmatrix}
  \sum_j L(x,0,j)^{-1}\\[-1mm]
  \vdots\\[-1mm]
  \sum_j L(x,r-1,j)^{-1}
 \end{pmatrix}
 \begin{pmatrix}
  \sum_k R(x,0,k)&\cdots&\sum_k R(x,r-1,k)
 \end{pmatrix}.
\end{equation}
Both factors have rank at most $n$, so their product has rank at most $n$.
Adding the difference in one matrix block increases the rank by at most
the rank of that difference. Subadditivity of rank now gives
\eqref{eq:upper}.
\end{proof}

\begin{proof}[Proof of Theorem~\ref{thm:obstruction}]
Recall that $H_x\in M_I(\F_q)$ and $Y_i,Z_x\in M_{rn}(\F_q)$.
We use the identifications
\[
 M_I(\F_q)\otimes_{\F_q}M_{rn}(\F_q)
 \cong M_I\!\bigl(M_{rn}(\F_q)\bigr)
 \cong M_{s^5rn}(\F_q).
\]
Thus the matrices have $s^5\times s^5$ blocks, each of size $rn\times rn$.
The symbol $\otimes$ denotes the Kronecker product: the $(j,k)$ block of
$H_x\otimes Z_x$ is $(H_x)_{jk}Z_x$.
We write $\bigoplus_{i\in I}Y_i$ for the block diagonal matrix with
diagonal blocks $Y_i$ and zero off-diagonal blocks.
From \eqref{eq:certificate}, we obtain the tensor identity of
\cite[equation (2.14)]{OpenAI}:
\begin{equation}\label{eq:tensor}
 -\bigoplus_{i\in I}Y_i=\sum_{x\in\F_q^5}H_x\otimes Z_x.
\end{equation}
Indeed, expanding $Z_x$ and changing the order of summation makes the
coefficient of $Y_i$ equal to $-E_{ii}$ for $i\in I$, and zero for
$i\in J\setminus I$.
The rank on the left of \eqref{eq:tensor} is the sum of the individual
ranks. Since $\rank H_x\le1$ by Lemma~\ref{lem:interpolation}, each
summand on the right satisfies
$\rank(H_x\otimes Z_x)=\rank H_x\,\rank Z_x\le\rank Z_x$.
We apply Lemma~\ref{lem:upper} at every point $x$. Each triangular face
$\tau=(x,i,u,v)$ has a unique associated point, so its defect is counted
exactly once. Rank subadditivity therefore gives
\[
 \begin{aligned}
 \sum_{i\in I}\rank_{\F_q}Y_i
 &\le\sum_{x\in\F_q^5}\rank_{\F_q}Z_x \le\sum_{x\in\F_q^5}
 \left(n+n\sum_{\substack{i\in J(x)\\u,v\in[r]}}\delta_{x,i,u,v}\right)
 =n\left(q^5+\sum_\tau\delta_\tau\right).
 \end{aligned}
\]
Using the second bound of Lemma~\ref{lem:lower} and cancelling the positive
real number $n$, we obtain
\[
 \frac{rs^5}{4}-q^5
 \le\sum_\tau\delta_\tau+\frac{r^2}{8}\sum_{i\in I}\lambda_i.
\]
Since $q<10^{23}$ and $rs^5=59049\cdot10^{111}>5\cdot10^{115}$,
we have $rs^5/4-q^5>q^5/4>q^5/5$. This proves \eqref{eq:linear}.
\end{proof}

\section{The negatively curved complex}\label{sec:geometry}

We equip every face of the complex $K$ from Section~\ref{sec:rank}
with the hyperbolic triangle of curvature $-1$ with angles
\begin{equation}\label{eq:angles}
 \theta_O=\pi/2,\qquad \theta_V=\theta_W=\pi/5.
\end{equation}
The angle sum is $9\pi/10<\pi$, so this triangle exists. Assign to all
$L$- and $R$-edges the equal side lengths opposite $\theta_W$ and
$\theta_V$, respectively, and to all $A$-edges the side length opposite
$\theta_O$. This includes edges contained in no face. All gluings are
therefore isometric.

An edge germ at a vertex is an incident edge together with that endpoint.
The angular link has one vertex for each edge germ and one edge for
each face corner. The length of a link edge is the angle
at the corresponding corner. A path in a graph is reduced if it has no
immediate edge reversal. A reduced closed path is cyclically reduced if the
last edge is not the inverse of the first edge. We set the distance between
directions in different link components equal to infinity.

We adapt the link argument of \cite[Lemma 2.3]{OpenAI} to the present
girth bound and angles.

\begin{lemma}\label{lem:links}
Every nonempty cyclically reduced closed path in any vertex link has
length at least $2\pi$. At either $V$ or $W$, two different $A$-directions
with the same index $i$ have distance at least $2\pi$, or belong to
different components.
\end{lemma}
\begin{proof}
At $O$, the link is simple and bipartite: a corner between
$L_{x,u,j}$ and $R_{x,v,k}$ determines $i\in J(x)$ uniquely by
$\alpha_x(i)=j-v$ and $\beta_x(i)=k-u$.
Every nonempty cyclically reduced closed path therefore has at least
four edges and length at least $4\theta_O=2\pi$.

We map the link at $V$ to $\Gamma$ by sending the direction $A_{i,u,v}$
to the line vertex $i$ and the direction $L_{x,u,j}$ to the point $x$.
Each corner maps to the corresponding incidence edge. We claim that
this is an immersion, that is, the map is injective on incident edges
at every vertex. Indeed, at $A_{i,u,v}$, a fixed point $z$ determines at
most one corner, namely that of the face indexed by $(z,i,u,v)$ if $i\in J(z)$.
At $L_{x,u,j}$, a fixed index $i$ determines
$v=j-\alpha_x(i)$ modulo $s_x$. There is at most one such $v\in[r]$,
since $s_x\ge r$. At $W$, the same argument uses
$u=k-\beta_x(i)$ modulo $t_x$ and $t_x\ge r$.

Under these immersions, both a nonempty cyclically reduced closed link
path and a reduced path joining distinct $A$-directions with the same
index $i$ map to nonempty closed walks in $\Gamma$ with no immediate
reversal at an internal vertex. Every such walk contains a simple cycle
and hence has at least ten edges. Since link edges at $V$ and $W$ have
length $\pi/5$, both assertions follow.
\end{proof}

\begin{proposition}\label{prop:cat}
The universal cover $\widetilde K$ is a proper $\CAT(-1)$ space. The
group $G=\pi_1(K,O)$ is torsion-free and word-hyperbolic, and $K$ is a
finite two-dimensional classifying space for $G$.
\end{proposition}
\begin{proof}
We first subdivide the complex so that cells meet in single faces.
Subdivide each edge at its midpoint and each face by joining an interior
point to its vertices and edge midpoints by geodesics. Since the closed
cells embed and each edge and face has a distinct new vertex, the
subdivision is a finite simplicial complex with the same intrinsic
metric. The angular link at a new face-interior vertex is a circle of
length $2\pi$. At an edge-interior vertex it consists of length-$\pi$
paths between the two edge directions.

The original angular links have the same metric after subdivision.
By Lemma~\ref{lem:links}, their injective circuits have length at least
$2\pi$, and all circuits in the new links have length at least
$2\pi$. It follows from the metric graph criterion
\cite[Lemma II.5.6]{BH} and the link condition
\cite[Theorem II.5.2]{BH} that $K$ is locally $\CAT(-1)$.

The lifted complex is locally finite and has finitely many shapes.
Its length metric is complete by \cite[Theorem I.7.13]{BH}. By local
compactness and \cite[Proposition I.3.7]{BH}, it is proper and geodesic.
The Cartan--Hadamard theorem \cite[Theorem II.4.1]{BH} now implies that the universal cover
$\widetilde K$ is globally $\CAT(-1)$. Deck transformations act freely,
properly and cocompactly by isometries. A finite group of isometries of
a complete $\CAT(0)$ space fixes a point
\cite[Corollary II.2.8]{BH}. Since the deck action is free, no
nonidentity deck transformation has finite order. The space is
hyperbolic \cite[Chapter III.H]{BH}, so the \v{S}varc--Milnor lemma
\cite[Proposition I.8.19]{BH} implies that $G$ is word-hyperbolic.
Since $\widetilde K$ is contractible, $K$ is a classifying space.
\end{proof}

It remains to show that the rectangular words studied before are all non-trivial in $G$. This is the content of the following lemma.
We adapt the argument following \cite[equation (2.6)]{OpenAI} to the
$\CAT(-1)$ metric constructed above.

\begin{lemma}\label{lem:nontrivial}
For $i\in I$, $u\ne u'$ and $v\ne v'$, the element
$w_{i;u,u',v,v'}$
is nonidentity in $G$.
\end{lemma}
\begin{proof}
The edge path
$
 p=A_{i,u,v}A_{i,u',v}^{-1}A_{i,u',v'}A_{i,u,v'}^{-1}
$
is a loop based at $V$, and $w_{i;u,u',v,v'}$ is the class of
the loop $c_Vpc_V^{-1}$ based at $O$.
At every turn of $p$, including the closing turn, the two
$A$-directions are distinct and have the same index $i$. By
Lemma~\ref{lem:links}, their link distance is greater than $\pi$.
It follows that $p$ is a positive-length closed local geodesic, by the
link description of angles \cite[Chapter I.7]{BH}. If this loop were
nullhomotopic, it would lift to a positive-length closed local geodesic
in $\widetilde K$. Every local geodesic in a $\CAT(-1)$ space is a
geodesic \cite[Proposition II.1.4]{BH}. This is impossible for a
positive-length closed path.
\end{proof}

\section{Proof of the main result}\label{sec:sofic}

Let $H$ be a sofic group and let $\eta>0$. Given finitely many words in
a finite alphabet of elements of $H$ and their inverses, choose a finite
set $F\subset H$ containing the alphabet, its inverses, all word prefixes,
and $1$. Let $\phi:F\to\Sym(n)$ satisfy the inequalities in the definition
of soficity from the introduction.
For a word of length $\ell\ge1$ with $k$ inverse letters, evaluate it
using the assigned permutations of the positive letters. Its distance
from the image under $\phi$ of the corresponding group element is at most $(\ell-1+k)\eta$.
Indeed, each inverse substitution changes the evaluated permutation
by less than $\eta$, since $\phi(a)\phi(a^{-1})$ is within $\eta$ of
$\id$. Each successive prefix multiplication introduces an error of
less than $\eta$. The asserted bound follows from bi-invariance and
the triangle inequality.

\begin{proof}[Proof of Theorem~\ref{thm:main}]
The geometric assertions follow from Proposition~\ref{prop:cat}.
By Lemma~\ref{lem:nontrivial}, all rectangular elements are nonidentity.
Assume that $G$ is sofic. Apply the preceding estimate in $G$
to all triangular and rectangular words, with the elements
$g_e$ as the alphabet. Assign $\phi(g_e)$ to each edge symbol.
Each triangular word has length three, one inverse letter, and group
value $1$. Hence $\delta_\tau<3\eta$. Each rectangular word has length
four and two inverse letters, so its evaluated permutation is within
$5\eta$ of $\phi(w)$. Since $w\ne1$, the latter permutation has
fixed-point fraction less than $\eta$. Thus the evaluated rectangle
has fixed-point fraction less than $6\eta$, and $\lambda_i<6\eta$.
These estimates also hold when some alphabet elements equal $1$.

Each point $x$ indexes $r^2|J(x)|$ triangular faces. By the
degree count in Section~\ref{sec:lines}, there are
$r^2((q-1)q^5-s^5)$ triangle indices in total. The data are fixed
independently of $n$ and $\eta$.
Applying \eqref{eq:linear} from Theorem~\ref{thm:obstruction} to the
permutations just constructed, and using the $s^5$ distinguished lines,
gives
\[
 \begin{aligned}
 \frac{q^5}{5}&\le\sum_\tau\delta_\tau+\frac{r^2}{8}\sum_{i\in I}\lambda_i\le 3r^2((q-1)q^5-s^5)\eta+\frac{r^2}{8}\cdot6s^5\eta
 \le 3r^2q^6\eta.
 \end{aligned}
\]
Choosing $0<\eta<5\cdot10^{-37}$ gives a contradiction to our choices for $q$ and $r$.
\end{proof}

\end{document}